%% file: main.tex
\documentclass[12pt]{article}
\usepackage{amsmath,amsfonts,amssymb,amsthm,mathtools,bm}
\usepackage[T1]{fontenc}
\usepackage[margin=3cm]{geometry}
\usepackage{setspace}
\usepackage{graphicx}
\usepackage{standalone}
\usepackage{float}
\usepackage{thmtools}
\usepackage{tikz}
\usepackage{stmaryrd}
\usepackage{hyperref}
\usepackage[capitalise]{cleveref}
\usepackage[skip=10pt, indent=20pt]{parskip}
\makeatletter
\def\thm@space@setup{%
  \thm@preskip=\parskip \thm@postskip=0pt
}
\makeatother

\hypersetup{
  colorlinks   = true,
  urlcolor     = blue,
  linkcolor    = black,
  citecolor   = blue
}

\DeclarePairedDelimiter{\abs}{\lvert}{\rvert}

\usepackage{xcolor}

\newcommand{\norm}[1]{\left\Vert #1 \right\Vert}
\newcommand{\ceil}[1]{\left\lceil #1 \right\rceil}

\newcommand{\set}[1]{\{#1\}}

\newcommand*{\cG}{\mathcal{G}}
\newcommand{\defn}[1]{\textit{#1}}

\newtheorem{theorem}{Theorem}[section]
\newtheorem{lemma}[theorem]{Lemma}

\renewcommand{\ge}{\geqslant}
\renewcommand{\le}{\leqslant}
\renewcommand{\geq}{\geqslant}
\renewcommand{\leq}{\leqslant}

\newcommand\mybox{\mathbin{\text{\scalebox{.84}{$\square$}}}}

\title{Tight Bounds for Hypercube Minor-Universality}
\author{
Emma Hogan\footnote{Mathematical Institute, University of Oxford, Oxford OX2 6GG, UK,\\ \texttt{\{hogan, michel, scott, tamitegama, jane.tan, tsarev\}@maths.ox.ac.uk}} \footnote{Supported by EPSRC grant EP/W524311/1.}\qquad
Lukas Michel\protect\footnotemark[1]\qquad
Alex Scott\protect\footnotemark[1] \footnote{Supported by EPSRC grant EP/X013642/1.}\\
Youri Tamitegama\protect\footnotemark[1]\qquad
Jane Tan\protect\footnotemark[1] \footnote{All Souls College, University of Oxford}\qquad
Dmitry Tsarev\protect\footnotemark[1]}
\date{}

\begin{document}

\maketitle

\begin{abstract}
Benjamini, Kalifa and Tzalik recently proved that there is an absolute constant $c>0$ such that any graph with at most $c\cdot2^d/d$ edges and no isolated vertices is a minor of the $d$-dimensional hypercube $Q_d$, while there is an absolute constant $K > 0$ such that $Q_d$ is not $(K\cdot2^d/\sqrt{d})$-minor-universal.  We show that $Q_d$ does not contain 3-uniform expander graphs with $C\cdot2^d/d$ edges as minors.  This matches the lower bound up to a constant factor and answers one of their questions.
\end{abstract}

\section{Introduction}

A (finite or infinite) graph is universal for a class $\cG$ of graphs if it contains every graph from $\cG$. Different instances of graph universality arise by specifying the notion of containment and class of graphs. In the infinite case, the question is often whether a universal graph exists; in the finite case, there is considerable work on finding small universal graphs, and such results are significant in extremal graph theory, structural graph theory, and various algorithmic problems.

Results on universal graphs date back to the early 1960s, when Rado \cite{rado64} proved that there exists a countable graph that contains every countable graph as an induced subgraph. There is now substantial literature with results on many graph classes with respect to subgraph or induced subgraph universality~\cite{Alon17, AN17, ADK17, BCEGS, BGP22, Chung90, KMP88, LR07}. There are also interesting negative results: answering a question of Ulam, a classical theorem of Pach \cite{Pach81} from 1981 states that there does not exist a countable planar graph that contains every countable planar graph as a subgraph. 

A different picture emerges for minor containment. For graphs $H$ and $G$, we say that $H$ is a \defn{minor} of $G$ if $H$ can be obtained from $G$ by a finite sequence of edge and vertex deletions, and edge contractions.
In this case, the analogue of Pach's theorem does not hold: in 1999, Diestel and K\"{u}hn \cite{DK99} constructed a countable planar graph that contains every countable planar graph as a minor. The investigation of minor-universality has been fruitful (see \cite{Georgakopoulos25} for some recent developments). A notable example is the theorem of Robertson, Seymour and Thomas \cite{RST94} that the $2n \times 2n$ grid is minor-universal for planar graphs on $n$ vertices, which has important consequences for bounding treewidth.

In this paper, we are concerned with minors of hypercubes. Hypercubes are a natural candidate graph in which to embed other graphs for various computational reasons (see, for example, \cite{Wu85},  which studies universality of hypercubes containing $k$-ary trees as subgraphs).  Very recently, Benjamini, Kalifa and Tzalik~\cite{BKT25} investigated the minor-universality of hypercubes.
They define a graph $G$ to be \emph{$m$-minor-universal} if every graph $H$ with at most $m$ edges and no isolated vertices is contained as a minor in $G$. Let $Q_d$ denote the $d$-dimensional hypercube whose vertices are binary strings of length $d$ where two vertices are adjacent whenever their Hamming distance is one. Benjamini, Kalifa and Tzalik's main result is the following.

\begin{theorem}[Theorem A in \cite{BKT25}]\label{bktthm}
    The hypercube $Q_d$ is $\Omega(\frac{2^d}{d})$-minor-universal. Moreover, there is an absolute constant $K > 0$ such that $Q_d$ is not $\frac{K\cdot2^d}{\sqrt{d}}$-minor-universal.
\end{theorem}

The upper and lower bounds differ by a factor of $\sqrt{d}$.
Benjamini, Kalifa and Tzalik went on to ask which (if either) of these two bounds is tight \cite[Question 1]{BKT25}.
We resolve this question by exhibiting a graph with $C\cdot 2^d/d$ edges which is not a minor of $Q_d$, which shows that the lower bound is in fact tight. 
For completeness, we also give a short version of the lower-bound proof from \cite{BKT25}.

\section{Short proof of lower bound}

We present a substantially shorter version of the lower bound argument from \cite{BKT25}. Given a permutation $\sigma$ of $X = [n_1] \times \dotsb \times [n_d]$, we say that $\sigma$ is \textit{one-dimensional in direction $i$} if, for all $x \in X$, we have that $\sigma(x)$ and $x$ may differ only in coordinate $i$. We begin with a short proof of a key lemma, stating that permutations of a $d$-dimensional grid can be decomposed into a linear number of one-dimensional permutations.

\begin{lemma}\label{lem:permut}
    Let $X = [n_1]\times\dotsb\times[n_d]$ and let $\sigma$ be a permutation of $X$. Then $\sigma$ can be written as
    \[
        \sigma = \sigma_{2d-1} \circ \dotsb \circ \sigma_1,
    \]
    where each $\sigma_i$ is a one-dimensional permutation of $X$ in direction $j_i = \abs{d - i} + 1$.
\end{lemma}

\begin{proof}
    We first sketch the idea. Thinking of the coordinate $d$ direction as `vertical', and all other directions as `horizontal', we can view $X$ as either $n_1 \dotsb n_{d-1}$ one-dimensional columns, or $n_d$ horizontal layers.  We construct the one-dimensional permutations in three steps: first, choose $\sigma_{1}$ to permute the entries in each column so that each layer consists of vectors whose images under $\sigma$ are in different columns of $X$; then, inductively find permutations $\sigma_2, \dots, \sigma_{2d-2}$ that rearrange the vectors in each layer so that every vector lies in the correct column; finally, choose $\sigma_{2d-1}$ to permute each column so that every vector is in the correct layer. 

    We proceed by induction on $d$. Clearly, the statement holds for $d = 1$, so suppose that $d \ge 2$ and that the statement holds for $d-1$. Let $Y = [n_1] \times \dots \times [n_{d-1}]$, and denote by $\pi_d(x) = (x_1, \dots, x_{d-1})$ the projection that forgets coordinate $d$.
    We construct a bipartite multigraph $G$ with bipartition $V(G) = U \cup V$ where $U$ and $V$ are both copies of $Y$, and for each $x \in X$ we add an edge $e_x$ from $\pi_d(x) \in U$ to $\pi_d(\sigma(x)) \in V$. Note that $G$ is $n_d$-regular. Therefore, by repeatedly applying Hall's theorem we can decompose $G$ into $n_d$ perfect matchings $M_1, \dots, M_{n_d}$. For $x \in X$, let $\ell_x \in [n_d]$ be such that $e_x \in M_{\ell_x}$.

    Define $\sigma_1\colon X \to X$ by $\sigma_1(x) = (\pi_d(x), \ell_x)$. Observe that $\sigma_1$ is one-dimensional in direction $d$. Also, this is a permutation because if $x, y \in X$ are distinct and satisfy $\pi_d(x) = \pi_d(y)$, then $e_x$ and $e_y$ are incident in $G$, so they cannot both belong to the same matching. This implies that $\ell_x \neq \ell_y$ and so $\sigma_1(x) \neq \sigma_1(y)$.
    Moreover, if $x, y \in X$ are distinct and satisfy $\pi_d(\sigma(x)) = \pi_d(\sigma(y))$, then $e_x$ and $e_y$ are again incident in $G$ and so $\ell_x \neq \ell_y$. Thus, for any $\ell \in [n_d]$ and any distinct $x, y \in X_\ell = \{x \in X : \ell_x = \ell\}$ we have $\pi_d(x) \neq \pi_d(y)$ and $\pi_d(\sigma(x)) \neq \pi_d(\sigma(y))$. This implies that $\tau_\ell\colon Y \to Y$ defined by $\tau_\ell(\pi_d(x)) = \pi_d(\sigma(x))$ for $x \in X_\ell$ is a permutation of $Y$. By induction, it follows that $\tau_\ell$ can be written as $\tau_{\ell,2d-2} \circ \dots \circ \tau_{\ell,2}$ where each $\tau_{\ell, i}$ is one-dimensional in direction $j_i$.

    For any $2 \le i \le 2d-2$, let $\sigma_i(y,\ell) = (\tau_{\ell,i}(y),\ell)$ which is one-dimensional in direction $j_i$, and let $\sigma' = \sigma_{2d-2} \circ \dots \circ \sigma_2$. Note that $\sigma'(y,\ell) = (\tau_\ell(y),\ell)$ which implies that $\sigma'(\sigma_1(x)) = \sigma'(\pi_d(x), \ell_x) = (\tau_{\ell_x}(\pi_d(x)), \ell_x) = (\pi_d(\sigma(x)), \ell_x)$ for all $x \in X$. Now define $\sigma_{2d-1}:X\to X$ so that $\sigma_{2d-1}(\pi_d(\sigma(x)), \ell_x) = \sigma(x)$. Then $\sigma = \sigma_{2d-1} \circ \sigma' \circ \sigma_1$ and $\sigma_{2d-1}$ is one-dimensional in direction $d$, as required.
\end{proof}

Given two graphs $F_1$ and $F_2$, their \emph{Cartesian product} $F_1 \mybox F_2$ is the graph with vertices $V(F_1)\times V(F_2)$ and edges between $(u,v)$ and $(u',v')$ whenever either $u=u'$ and $vv'\in E(F_2)$ or $v=v'$ and $uu'\in E(F_1)$. We will focus on the cube for simplicity, but the argument generalises to products of other graphs as shown in \cite{BKT25}.

\begin{theorem}
    There is an absolute constant $c>0$ such that $Q_d$ is $\frac{c\cdot2^d}{d}$-minor-universal. 
\end{theorem}
 
\begin{proof}
We follow a strategy very similar to \cite{BKT25}.
Let $G$ be a graph with at most $c\cdot2^d/d$ edges and no isolated vertices. We need to show that $G$ is a minor of $Q_d$. Let $a = \ceil{\log_2(2\abs{E(G)})}$, and consider a copy of $Q_a\mybox C_{2d}$ in $Q_d$ where $Q_a$ uses only the first $a$ coordinates of $Q_d$ and $C_{2d}$ the next $\ceil{\log_2(2d)}$ coordinates. All vertices of this copy have the same value on all remaining coordinates.
Label the vertices of $C_{2d}$ with elements of $\{0,1,\dots,2d-1\}$ so that consecutive integers are adjacent. We will refer to this cycle as the \defn{temporal} dimension.

To each vertex $v\in V(G)$, assign a path $P(v)$ in $Q_a \mybox \{0\}$ with $\abs{N(v)}$ vertices such that $\set{P(v):v\in V(G)}$ is a collection of vertex-disjoint paths. This is possible since $Q_a$ contains a Hamiltonian cycle with $2^a \ge 2 \abs{E(G)}$ vertices which can be split into these paths. Index the vertices in $P(v)$ arbitrarily by vertices of $N(v)$, so we can write $V(P(v))=\{v_x: x\in N(v)\}$.

It now suffices to find a collection of vertex-disjoint paths joining the vertices $x_y$ and $y_x$ for each edge $xy\in E(G)$ with internal vertices outside of $Q_a\mybox \{0\}$. 
Let $\sigma$ be a permutation of $Q_a\mybox \set{0}$ which satisfies $\sigma(x_y) = y_x$ for each $xy\in E(G)$.
By \autoref{lem:permut}, we may write $\sigma = \sigma_{2d-1} \circ \dotsb \circ \sigma_1$ where each $\sigma_i$ is one-dimensional.

For an edge $xy\in E(G)$, we define a path between $x_0 = x_y$ and $x_{2d-1} = y_x$ piecewise as follows.
Consider the sequence of vertices $(x_0, 0), (x_1,1),\dotsc,(x_{2d-1}, 2d-1), (x_{2d-1}, 0)$ where $x_i = \sigma_i(x_{i-1})$ for each $i>0$.
Fix an $i>0$.
We can go from $(x_{i-1}, i-1)$ to $(x_{i}, i)$ by first moving in the temporal dimension, and then in the unique coordinate in which $x_{i-1}$ and $x_i$ differ (if they differ); call this path $P_x$. Consider an analogously defined path $P_z$ between $(z_{i-1}, i-1)$ and $(z_i,i)$. If $P_x$ and $P_z$ are not already disjoint, it must be that $z_i = \sigma_i(z_{i-1}) = x_{i-1}$, and hence also that $x_i = \sigma_i(x_{i-1}) = z_{i-1}$. In this case, from the last four coordinates of $Q_d$ choose any two coordinates which were not used in time step $i-1$, say coordinates $d-1$ and $d$.
Starting from the first vertices of the paths, flip the bits $d-1$ and $d$ respectively, then follow $P_x$ and $P_z$ respectively, and finally flip back the bits $d-1$ and $d$ (see \autoref{fig:paths}).
This ensures that the paths $P_x$ and $P_z$ are traversed in distinct copies of $Q_a\mybox C_{2d}$ and therefore do not intersect.
Moreover, the moves in coordinates $d-1$ and $d$ do not intersect the existing paths at time step $i-1$ as we chose two coordinates which were not used in that time step.

Since the paths defined above are vertex-disjoint for each time step, they are also vertex-disjoint in $Q_d$. Thus, stitching them together produces a path between $x_y$ and $y_x$ such that these paths for all $xy\in E(G)$ are pairwise vertex-disjoint as required.
\end{proof}

\begin{figure}[h]
     \centering
     \scalebox{0.85}{\input{permutation-picture-final.tex}}
     \caption{Constructing paths $(x_{i-1},i-1)$\textrightarrow$(x_{i},i)$ and $(z_{i-1},i-1)$\textrightarrow$(z_{i},i)$ when $z_i=x_{i-1}$.}
     \label{fig:paths}
 \end{figure}
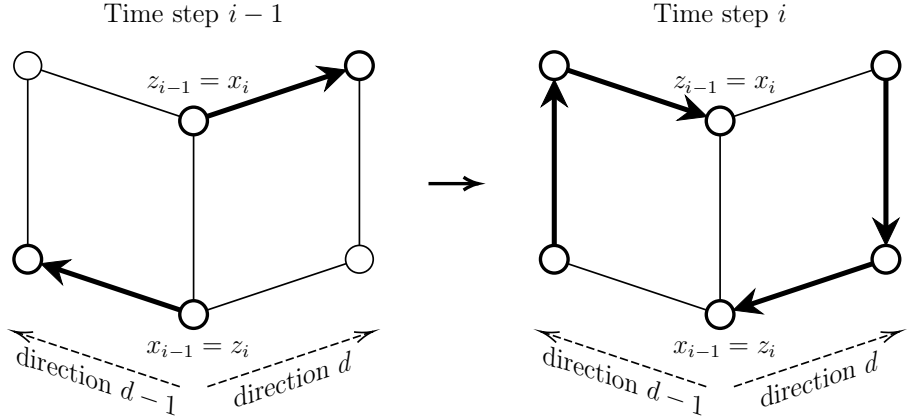

\section{Tight upper bound}

Benjamini, Kalifa and Tzalik \cite{BKT25} showed that there is an absolute constant $K > 0$ such that $Q_d$ is not $\frac{K\cdot2^d}{\sqrt{d}}$-minor-universal.  In this section we improve this by a $\sqrt d$ factor, to obtain an upper bound that matches the lower bound of \autoref{bktthm} to within a constant factor.  As in \cite{BKT25}, we work with expander graphs, which are a natural candidate for graphs that should be `difficult' to embed.

\begin{theorem}
There is an absolute constant $C>0$ such that $Q_d$ is not $\frac{C\cdot 2^d}{d}$-minor-universal.
\end{theorem}

\begin{proof} 
Let $G$ be a 3-regular graph with $2n$ vertices, where $2n \in [45\cdot 2^d / d,50\cdot 2^d/d]$, that satisfies the following expansion property: for any set $S\subseteq V(G)$ of at most $n$ vertices, $\abs{N(S)} \geq 0.18 \abs{S}$, where $N(S)$ denotes the neighbours of $S$ in $V(G) \setminus S$.
It is well-known that such a graph exists (in fact, almost every 3-regular graph satisfies this property; see \cite[Theorem 1]{Bo88}, and observe that 0.18 satisfies the condition in the theorem when $r=3$). Note that $G$ has $3n \leq 100\cdot 2^d / d$ edges, and so it suffices to show that $G$ is not a minor of $Q_d$.

Suppose for a contradiction that $Q_d$ contains $G$ as a minor. Since $G$ is 3-regular, it is easy to see that this implies that $Q_d$ contains a subgraph $G'$ which is a subdivision of $G$. We will show that there are not enough vertices in $Q_d$ to support such a subdivision.

For every vertex $v \in V(G)$, let $s_v \in V(G') \subseteq \set{0,1}^d$ be the corresponding vertex in the subdivision $G'$, and let $S = \{s_v : v \in V(G)\}$ (so $\abs{S} = 2n$). For any edge $uv\in E(G)$, the length of the corresponding $(s_u, s_v)$-path in $G'$ is at least the Hamming distance $\norm{s_u - s_v}_1$ between $s_u$ and $s_v$. Since these paths are edge-disjoint, the sum of the corresponding Hamming distances over all edges in $G$ counts every edge of $G'$ at most once. That is, 
\[
 \abs{V(G')} = \abs{E(G')} - \abs{S} / 2 \geq \sum_{uv\in E(G)} \norm{s_u - s_v}_1 - \abs{S} / 2 = \sum_{uv\in E(G)} \norm{s_u - s_v}_1 - n.
\]

For each $i \in [d]$, let $S_i$ be the smaller set out of $\set{s \in S : s_i = 0}$ and $\set{s \in S : s_i = 1}$, and let $E_i$ denote the set of edges $uv \in E(G)$ such that $s_u\in S_i$ and $s_v \in S\setminus S_i$ (i.e., such that $s_u$ and $s_v$ differ in bit $i$). Since $\abs{S_i} \le \abs{V(G)} / 2$, the expansion property of $G$ implies that $\abs{E_i} \ge 0.18 \abs{S_i}$. Note that for any edge $uv\in E(G)$, its contribution to $\sum_i \abs{E_i}$ is precisely the Hamming distance between $s_u$ and $s_v$, and so 
\[
\sum_{uv\in E(G)} \norm{s_u - s_v}_1 = \sum_{i = 1}^d \abs{E_i} \ge 0.18 \sum_{i = 1}^d \abs{S_i}.
\]
Without loss of generality, we may assume that $S_i=\set{s \in S : s_i = 1}$ for all $i \in [d]$. Then $\sum_i \abs{S_i} = \sum_{s \in S} \norm{s}_1$. When $d$ is sufficiently large, the number of binary strings with at most $d/4$ ones is much less than $2^d/d$, which is less than $\abs{S}/2$. Thus $S$ contains at least $\abs{S}/2$ elements with more than $d/4$ ones, and so we have $\sum_{s \in S} \norm{s}_1 > d\abs{S}/8$.

Therefore, when $d$ is sufficiently large, we have
\[
\abs{V(G')} \geq 0.18\sum_{i = 1}^d \abs{S_i} - n > \frac{ 0.18\cdot d \cdot 45\cdot 2^d}{8 d} - \frac{50\cdot 2^d}{2d} \geq 2^d = \abs{V(Q^d)}.
\]
This contradicts the assumption that $Q^d$ contains $G'$ as a subgraph.
\end{proof}

\bibliographystyle{style}
\bibliography{references}
\end{document}

%% file: permutation-picture-final.tex
\tikzset{every picture/.style={line width=0.75pt}} 

\begin{tikzpicture}[x=0.75pt,y=0.75pt,yscale=-1,xscale=1]

\draw    (140.3,85.68) -- (140.3,200.63) ;
\draw [line width=0.75]    (41.9,52.84) -- (41.9,167.79) ;
\draw    (238.69,52.84) -- (238.69,167.79) ;
\draw    (41.9,52.84) -- (140.3,85.68) ;
\draw    (41.9,167.79) -- (140.3,200.63) ;
\draw    (238.69,52.84) -- (140.3,85.68) ;
\draw    (238.69,167.79) -- (140.3,200.63) ;
\draw  [fill={rgb, 255:red, 255; green, 255; blue, 255 }  ,fill opacity=1 ] (33.7,52.84) .. controls (33.7,48.31) and (37.37,44.63) .. (41.9,44.63) .. controls (46.43,44.63) and (50.1,48.31) .. (50.1,52.84) .. controls (50.1,57.38) and (46.43,61.05) .. (41.9,61.05) .. controls (37.37,61.05) and (33.7,57.38) .. (33.7,52.84) -- cycle ;
\draw  [fill={rgb, 255:red, 255; green, 255; blue, 255 }  ,fill opacity=1 ][line width=1.5]  (132.1,85.68) .. controls (132.1,81.15) and (135.77,77.47) .. (140.3,77.47) .. controls (144.82,77.47) and (148.5,81.15) .. (148.5,85.68) .. controls (148.5,90.22) and (144.82,93.89) .. (140.3,93.89) .. controls (135.77,93.89) and (132.1,90.22) .. (132.1,85.68) -- cycle ;
\draw  [fill={rgb, 255:red, 255; green, 255; blue, 255 }  ,fill opacity=1 ][line width=1.5]  (33.7,167.79) .. controls (33.7,163.25) and (37.37,159.58) .. (41.9,159.58) .. controls (46.43,159.58) and (50.1,163.25) .. (50.1,167.79) .. controls (50.1,172.32) and (46.43,176) .. (41.9,176) .. controls (37.37,176) and (33.7,172.32) .. (33.7,167.79) -- cycle ;
\draw  [fill={rgb, 255:red, 255; green, 255; blue, 255 }  ,fill opacity=1 ][line width=1.5]  (132.1,200.63) .. controls (132.1,196.1) and (135.77,192.42) .. (140.3,192.42) .. controls (144.82,192.42) and (148.5,196.1) .. (148.5,200.63) .. controls (148.5,205.17) and (144.82,208.84) .. (140.3,208.84) .. controls (135.77,208.84) and (132.1,205.17) .. (132.1,200.63) -- cycle ;
\draw  [fill={rgb, 255:red, 255; green, 255; blue, 255 }  ,fill opacity=1 ][line width=1.5]  (230.49,52.84) .. controls (230.49,48.31) and (234.16,44.63) .. (238.69,44.63) .. controls (243.22,44.63) and (246.89,48.31) .. (246.89,52.84) .. controls (246.89,57.38) and (243.22,61.05) .. (238.69,61.05) .. controls (234.16,61.05) and (230.49,57.38) .. (230.49,52.84) -- cycle ;
\draw  [fill={rgb, 255:red, 255; green, 255; blue, 255 }  ,fill opacity=1 ] (230.49,167.79) .. controls (230.49,163.25) and (234.16,159.58) .. (238.69,159.58) .. controls (243.22,159.58) and (246.89,163.25) .. (246.89,167.79) .. controls (246.89,172.32) and (243.22,176) .. (238.69,176) .. controls (234.16,176) and (230.49,172.32) .. (230.49,167.79) -- cycle ;
\draw  [dash pattern={on 3.75pt off 1.5pt}]  (132.1,243.47) -- (35.6,211.26) ;
\draw [shift={(33.7,210.63)}, rotate = 18.46] [color={rgb, 255:red, 0; green, 0; blue, 0 }  ][line width=0.75]    (10.93,-3.29) .. controls (6.95,-1.4) and (3.31,-0.3) .. (0,0) .. controls (3.31,0.3) and (6.95,1.4) .. (10.93,3.29)   ;
\draw  [dash pattern={on 3.75pt off 1.5pt}]  (148.5,243.47) -- (244.99,211.26) ;
\draw [shift={(246.89,210.63)}, rotate = 161.54] [color={rgb, 255:red, 0; green, 0; blue, 0 }  ][line width=0.75]    (10.93,-3.29) .. controls (6.95,-1.4) and (3.31,-0.3) .. (0,0) .. controls (3.31,0.3) and (6.95,1.4) .. (10.93,3.29)   ;
\draw    (452.7,85.68) -- (452.7,200.63) ;
\draw    (354.31,52.84) -- (354.31,167.79) ;
\draw    (551.1,52.84) -- (551.1,167.79) ;
\draw    (354.31,52.84) -- (452.7,85.68) ;
\draw    (354.31,167.79) -- (452.7,200.63) ;
\draw    (551.1,52.84) -- (452.7,85.68) ;
\draw    (551.1,167.79) -- (452.7,200.63) ;
\draw  [fill={rgb, 255:red, 255; green, 255; blue, 255 }  ,fill opacity=1 ][line width=1.5]  (346.11,52.84) .. controls (346.11,48.31) and (349.78,44.63) .. (354.31,44.63) .. controls (358.84,44.63) and (362.51,48.31) .. (362.51,52.84) .. controls (362.51,57.38) and (358.84,61.05) .. (354.31,61.05) .. controls (349.78,61.05) and (346.11,57.38) .. (346.11,52.84) -- cycle ;
\draw  [fill={rgb, 255:red, 255; green, 255; blue, 255 }  ,fill opacity=1 ][line width=1.5]  (444.5,85.68) .. controls (444.5,81.15) and (448.18,77.47) .. (452.7,77.47) .. controls (457.23,77.47) and (460.9,81.15) .. (460.9,85.68) .. controls (460.9,90.22) and (457.23,93.89) .. (452.7,93.89) .. controls (448.18,93.89) and (444.5,90.22) .. (444.5,85.68) -- cycle ;
\draw  [fill={rgb, 255:red, 255; green, 255; blue, 255 }  ,fill opacity=1 ][line width=1.5]  (346.11,167.79) .. controls (346.11,163.25) and (349.78,159.58) .. (354.31,159.58) .. controls (358.84,159.58) and (362.51,163.25) .. (362.51,167.79) .. controls (362.51,172.32) and (358.84,176) .. (354.31,176) .. controls (349.78,176) and (346.11,172.32) .. (346.11,167.79) -- cycle ;
\draw  [fill={rgb, 255:red, 255; green, 255; blue, 255 }  ,fill opacity=1 ][line width=1.5]  (444.5,200.63) .. controls (444.5,196.1) and (448.18,192.42) .. (452.7,192.42) .. controls (457.23,192.42) and (460.9,196.1) .. (460.9,200.63) .. controls (460.9,205.17) and (457.23,208.84) .. (452.7,208.84) .. controls (448.18,208.84) and (444.5,205.17) .. (444.5,200.63) -- cycle ;
\draw  [fill={rgb, 255:red, 255; green, 255; blue, 255 }  ,fill opacity=1 ][line width=1.5]  (542.9,52.84) .. controls (542.9,48.31) and (546.57,44.63) .. (551.1,44.63) .. controls (555.63,44.63) and (559.3,48.31) .. (559.3,52.84) .. controls (559.3,57.38) and (555.63,61.05) .. (551.1,61.05) .. controls (546.57,61.05) and (542.9,57.38) .. (542.9,52.84) -- cycle ;
\draw  [fill={rgb, 255:red, 255; green, 255; blue, 255 }  ,fill opacity=1 ][line width=1.5]  (542.9,167.79) .. controls (542.9,163.25) and (546.57,159.58) .. (551.1,159.58) .. controls (555.63,159.58) and (559.3,163.25) .. (559.3,167.79) .. controls (559.3,172.32) and (555.63,176) .. (551.1,176) .. controls (546.57,176) and (542.9,172.32) .. (542.9,167.79) -- cycle ;
\draw  [dash pattern={on 3.75pt off 1.5pt}]  (444.5,243.47) -- (348.01,211.26) ;
\draw [shift={(346.11,210.63)}, rotate = 18.46] [color={rgb, 255:red, 0; green, 0; blue, 0 }  ][line width=0.75]    (10.93,-3.29) .. controls (6.95,-1.4) and (3.31,-0.3) .. (0,0) .. controls (3.31,0.3) and (6.95,1.4) .. (10.93,3.29)   ;
\draw  [dash pattern={on 3.75pt off 1.5pt}]  (460.9,243.47) -- (557.4,211.26) ;
\draw [shift={(559.3,210.63)}, rotate = 161.54] [color={rgb, 255:red, 0; green, 0; blue, 0 }  ][line width=0.75]    (10.93,-3.29) .. controls (6.95,-1.4) and (3.31,-0.3) .. (0,0) .. controls (3.31,0.3) and (6.95,1.4) .. (10.93,3.29)   ;
\draw [line width=2.25]    (148.09,83.06) -- (226.16,57.05) ;
\draw [shift={(230.9,55.47)}, rotate = 161.58] [fill={rgb, 255:red, 0; green, 0; blue, 0 }  ][line width=0.08]  [draw opacity=0] (16.07,-7.72) -- (0,0) -- (16.07,7.72) -- (10.67,0) -- cycle    ;
\draw [line width=2.25]    (132.51,198.02) -- (54.43,171.98) ;
\draw [shift={(49.69,170.4)}, rotate = 18.45] [fill={rgb, 255:red, 0; green, 0; blue, 0 }  ][line width=0.08]  [draw opacity=0] (16.07,-7.72) -- (0,0) -- (16.07,7.72) -- (10.67,0) -- cycle    ;
\draw [line width=2.25]    (440.17,81.47) -- (362.1,55.43) ;
\draw [shift={(444.91,83.06)}, rotate = 198.45] [fill={rgb, 255:red, 0; green, 0; blue, 0 }  ][line width=0.08]  [draw opacity=0] (16.07,-7.72) -- (0,0) -- (16.07,7.72) -- (10.67,0) -- cycle    ;
\draw [line width=2.25]    (465.24,196.42) -- (543.31,170.42) ;
\draw [shift={(460.49,198)}, rotate = 341.58] [fill={rgb, 255:red, 0; green, 0; blue, 0 }  ][line width=0.08]  [draw opacity=0] (16.07,-7.72) -- (0,0) -- (16.07,7.72) -- (10.67,0) -- cycle    ;
\draw [line width=2.25]    (551.1,61.05) -- (551.1,154.58) ;
\draw [shift={(551.1,159.58)}, rotate = 270] [fill={rgb, 255:red, 0; green, 0; blue, 0 }  ][line width=0.08]  [draw opacity=0] (16.07,-7.72) -- (0,0) -- (16.07,7.72) -- (10.67,0) -- cycle    ;
\draw [line width=2.25]    (354.31,159.58) -- (354.31,66.05) ;
\draw [shift={(354.31,61.05)}, rotate = 90] [fill={rgb, 255:red, 0; green, 0; blue, 0 }  ][line width=0.08]  [draw opacity=0] (16.07,-7.72) -- (0,0) -- (16.07,7.72) -- (10.67,0) -- cycle    ;
\draw [line width=1.5]    (279.69,122.89) -- (309.49,122.89) ;
\draw [shift={(309,122.89)}, rotate = 180] [color={rgb, 255:red, 0; green, 0; blue, 0 }  ][line width=1.5]    (8.53,-3.82) .. controls (5.42,-1.79) and (2.58,-0.52) .. (0,0) .. controls (2.58,0.52) and (5.42,1.79) .. (8.53,3.82)   ;

\draw (110.94,58) node [anchor=north west][inner sep=0.75pt]  [font=\normalsize]  {$z_{i}{}_{-1} =x_{i}$};
\draw (110.94,215) node [anchor=north west][inner sep=0.75pt]  [font=\normalsize]  {$x_{i}{}_{-1} =z_{i}$};
\draw (82.9,240) node  [font=\normalsize,rotate=-18.44,xslant=-0.27] [align=left] {direction $d-1$};
\draw (198.89,239) node  [font=\normalsize,rotate=-341.6,xslant=0.27] [align=left] {direction $d$};
\draw (140.3,22.85) node    {Time step $i-1$};
\draw (423.34,58) node [anchor=north west][inner sep=0.75pt]  [font=\normalsize]  {$z_{i}{}_{-1} =x_{i}$};
\draw (423.34,215) node [anchor=north west][inner sep=0.75pt]  [font=\normalsize]  {$x_{i}{}_{-1} =z_{i}$};
\draw (395.31,240) node  [font=\normalsize,rotate=-18.44,xslant=-0.27] [align=left] {direction $d-1$};
\draw (511.3,239) node  [font=\normalsize,rotate=-341.6,xslant=0.27] [align=left] {direction $d$};
\draw (452.7,22.85) node    {Time step $i$};

\end{tikzpicture}